\documentclass[a4paper,10pt]{article}

\usepackage{amsfonts}
\usepackage{amsthm}
\usepackage{amsmath}
\usepackage{amssymb}
\usepackage{mathtools} %for prescript

\usepackage[top=1in, bottom=1in, left=1in, right=1in]{geometry}
\usepackage{xcolor}

\usepackage{graphicx}
\usepackage{epstopdf}% To incorporate .eps illustrations using PDFLaTeX, etc.
\usepackage{subfigure}% Support for small, `sub' figures and tables

\usepackage{authblk}

\theoremstyle{plain}% Theorem-like structures
\newtheorem{theorem}{Theorem}[section]

\newtheorem{lemma}[theorem]{Lemma}
\newtheorem{proposition}[theorem]{Proposition}

\theoremstyle{definition}
\newtheorem{definition}[theorem]{Definition}

\theoremstyle{remark}
\newtheorem{remark}[theorem]{Remark}

\DeclareSymbolFont{pxfontssymbolsC}{U}{pxsyc}{m}{n}
\DeclareMathSymbol{\coloneqq}{\mathrel}{pxfontssymbolsC}{66}

\usepackage{xcolor}
\usepackage{soul}
\definecolor{afcol}{rgb}{1,0,0}

\providecommand{\Real}{\mathop{\rm Re}\nolimits}%
\providecommand{\Imag}{\mathop{\rm Im}\nolimits}%
\providecommand{\keywords}[1]{\textbf{\textit{Keywords:}} #1}

\begin{document}

%\jvol{00} \jnum{00} \jyear{2015} \jmonth{February}

\title{A complex analysis approach to Atangana--Baleanu fractional calculus}

\date{}

\author{Arran Fernandez\thanks{Email: \texttt{arran.fernandez@emu.edu.tr}}}

\affil{{\small Department of Mathematics, Faculty of Arts and Sciences, Eastern Mediterranean University, Gazimagusa, TRNC, Mersin 10, Turkey}}

\maketitle

\begin{abstract}
The standard definition for the Atangana--Baleanu fractional derivative involves an integral transform with a Mittag-Leffler function in the kernel. We show that this integral can be rewritten as a complex contour integral which can be used to provide an analytic continuation of the definition to complex orders of differentiation. We discuss the implications and consequences of this extension, including a more natural formula for the Atangana--Baleanu fractional integral and for iterated Atangana--Baleanu fractional differintegrals.
\end{abstract}


\section{Introduction} \label{Sec:intro}
Fractional calculus is the study of extending the concept of $n$th derivatives and $n$th integrals, for natural numbers $n$, to a concept of $\nu$th derivatives and $\nu$th integrals -- often considered together and called \textbf{differintegrals} -- for more general types of number $\nu$. This field of study has a long history \cite{dugowson,miller-ross,samko-kilbas-marichev} and has discovered many applications to real-world problems in fields including viscoelasticity, chaos theory, diffusion problems, dynamical systems, bioengineering, etc. \cite{bagley,herrmann,hilfer,hristov,mainardi,magin,tarasov,west}

In many applications, it is enough to consider the order of differintegration $\nu$ as a real number, or even a number in some finite interval such as $[0,1]$. These conventions are used for the most part in various sources including the standard textbooks \cite{miller-ross,oldham-spanier}. However, it is also interesting and important to consider fractional extensions of complex analysis. Most of the standard definitions of fractional $\nu$th differintegrals apply equally well for $\nu\in\mathbb{C}$ as for $\nu\in\mathbb{R}$. For example, the most commonly used definition, the \textbf{Riemann--Liouville} or \textbf{RL} one, is as follows:
\begin{alignat}{2}
\label{RLdef:int} \prescript{RL}{c}I^{\nu}_zf(z)&=\frac{1}{\Gamma(\nu)}\int_c^z(z-w)^{\nu-1}f(w)\,\mathrm{d}w,&&\quad\quad\Re(\nu)>0; \\
\label{RLdef:der} \prescript{RL}{c}D^{\nu}_zf(z)&=\frac{\mathrm{d^n}}{\mathrm{d}z^n}\left(\prescript{RL}{c}I^{n-\nu}_zf(z)\right),\quad n\coloneqq\lfloor\Re(\nu)\rfloor+1,&&\quad\quad\Re(\nu)\geq0.
\end{alignat}
Here \eqref{RLdef:int} is the definition for fractional integrals and \eqref{RLdef:der} is the definition for fractional derivatives. By putting the two definitions together, we can allow $\nu$ to vary across the entire complex plane, writing \[\prescript{RL}{c}D^{-\nu}_zf(z)=\prescript{RL}{c}I^{\nu}_zf(z)\] for all $\nu\in\mathbb{C}$. Note that this gives a fractional differintegral $\prescript{RL}{c}D^{\nu}_zf(z)$ which is analytic as a function of $\nu$ for all $\nu\in\mathbb{C}$ \cite{samko-kilbas-marichev}.

A key problem in complex fractional calculus is the issue of branch points and branch cuts. The singular function $(z-w)^{\nu-1}$ which appears in the integrand of \eqref{RLdef:int} has a branch point at $w=z$, namely the point at one end of the contour of integration. It is necessary then to consider choices of branch for this function, in such a way as to yield a reasonable form for the expression \eqref{RLdef:int}. This issue is discussed e.g. in \cite[\S22]{samko-kilbas-marichev}.

A good way of resolving the branch cut issue is to redefine the Riemann--Liouville fractional differintegral in a way more suited to complex analysis. Namely, the following formula, which was first proposed in \cite{nekrassov} and further examined in \cite{osler,oldham-spanier,samko-kilbas-marichev}:
\begin{equation}
\label{RLdef:Cauchy}
\prescript{\mathbb{C}}{c}D^{\nu}_zf(z)=\frac{\Gamma(\nu+1)}{2\pi i}\int_H(w-
z)^{-\nu-1}f(w)\,\mathrm{d}w,\quad\quad\nu\in\mathbb{C}\backslash\mathbb{Z}^-,
\end{equation}
where $H$ denotes the finite Hankel contour which starts at the point $c$, wraps around $z$ in a counterclockwise sense, and returns to $c$. Here the infinite branch cut from $z$ is assumed to pass through $c$, so that the Hankel contour starts `below' the branch cut and finishes `above' it without ever crossing it. Formally, the Hankel contour can be defined as the union of the following three subcontours:
\begin{align}
\label{Hankel1} H_1&=\{z+re^{-i\pi}(z-c):1>r>\epsilon\}, \\
\label{Hankel2} H_2&=\{z+\epsilon e^{i\theta}(z-c):-\pi<\theta<\pi\}, \\
\label{Hankel3} H_3&=\{z+re^{i\pi}(z-c):\epsilon<r<1\},
\end{align}
where $\epsilon$ is a small positive constant (permitted to tend to zero by Cauchy's theorem) and $H_1$, $H_2$, $H_3$ are contiguous at the points $z+\epsilon(c-z)$. Here the notation of $e^{i\pi}$ and $e^{-i\pi}$ is used to indicate which side of the branch cut the contour is passing along.

This formula for fractional differintegration, which we shall call the \textbf{Cauchy} definition of fractional calculus since it is a natural extension of the Cauchy integral formula in complex analysis, can be compared with the Riemann--Liouville definition as follows:
\begin{itemize}
\item The Riemann--Liouville differintegral has a split-domain definition, requiring both \eqref{RLdef:int} and \eqref{RLdef:der} in order to be fully defined. Here the Cauchy differintegral has an advantage, because the same formula \eqref{RLdef:Cauchy} is valid and meromorphic on the whole complex plane.
\item The Riemann--Liouville differintegral can be applied to any $L^1$ function $f$, with differentiability conditions on $f$ in the case of fractional derivatives. Here the Cauchy differintegral has a disadvantage, because it is only useful when $f$ is analytic at least on a neighbourhood of $c$.
\item The Riemann--Liouville and Cauchy differintegrals are equivalent to each other:
\[\prescript{\mathbb{C}}{c}D^{\nu}_zf(z)=\prescript{RL}{c}D^{\nu}_zf(z)\]
for $f$ an analytic function \cite{nekrassov,oldham-spanier}. The proof for $\Re(\nu)<0$ involves collapsing the Hankel contour of integration in \eqref{RLdef:Cauchy} onto the branch cut, dividing the integral into three parts, and using the reflection formula for gamma functions to recover \eqref{RLdef:int}. Then the proof for $\Re(\nu)\geq0$ follows by analytic continuation in $\nu$.
\end{itemize}

All of the above is well-known from the standard textbooks on fractional calculus. However, the Riemann--Liouville formulae are not the only way to define fractional differintegrals. There are many different `models' of fractional calculus given by different differintegral formulae; some of them can be proven equivalent to each other, but others are completely separate. Some of the newer models of fractional calculus are justified by means of their real-life applications, while others are purely mathematical extensions of the idea of fractional differintegration. See for example \cite{fernandez-ozarslan-baleanu,garra-gorenflo-polito-tomovski,hristov2,sousa-oliveira,tateishi} for some discussion and examples of these newer models and their interrelationships.

The aim of the current work is to extend the idea of Cauchy fractional calculus, which is an equivalent formulation of the Riemann--Liouville model, and use similar ideas to propose equivalent complex-analytic formulations of other models of fractional calculus. This is important because there are many other useful and applicable models besides Riemann--Liouville, and a complex-analytic approach is sometimes invaluable for applications. For example, complex analysis has given rise to special methods for partial differential equations, such as the d-bar method which is still an active topic of research today \cite{ablowitz-fokas,fokas-pinotsis,fokas-vanderweele}.

In particular, we shall focus on the \textbf{Atangana--Baleanu} model, which we define in Section \ref{Sec:AB} below using Mittag-Leffler functions. Note that the Mittag-Leffler function has a strong and well-established connection with fractional calculus \cite{haubold-mathai-saxena,mathai-haubold}, and the Atangana--Baleanu model in particular has been proven useful by its many applications \cite{abro-khan-tassaddiq,bahaa1,bahaa2,bahaa3,bas-ozarslan,hristov}. However, so far the Atangana--Baleanu model has been defined only for real orders of differintegration, so a complex formulation is something lacking in the field. Thus, all existing literature on this type of fractional calculus has been restricted to the real domain only, unable to take advantage of the wealth of methods that can be found in complex analysis. This is a consequence of the definitions used so far which have not considered the complex domain as a possibility. The current paper brings a fresh complex-analytic perspective to the study of the Atangana--Baleanu fractional calculus, which opens the door for much new research in the future, for example using a complex analysis approach for Atangana--Baleanu fractional PDEs.

The structure of this paper is as follows. In Section \ref{Sec:AB} we provide a short introduction to the theory of the Atangana--Baleanu fractional calculus as it has been developed so far. In Section \ref{Sec:main} we derive rigorously the main definitions of the paper, checking carefully at every stage the validity of the extensions and formulae. More specifically, Section \ref{Sec:main1} is devoted to the Atangana--Baleanu derivatives (of both Riemann--Liouville and Caputo types) while Section \ref{Sec:main2} covers both the Atangana--Baleanu integrals and the iterated Atangana--Baleanu differintegrals. Section \ref{Sec:consequences} establishes some fundamental facts about the extended operators which follow naturally from the work in Section \ref{Sec:main} and the theory of analytic continuation. Section \ref{Sec:conclusions} provides the conclusions.

\section{Atangana--Baleanu fractional calculus} \label{Sec:AB}

The Atangana--Baleanu, or \textbf{AB}, definition of fractional calculus was first established in \cite{atangana-baleanu} and further investigated in works such as \cite{abdeljawad-baleanu,baleanu-fernandez,djida-atangana-area} and others. The definitions of the AB fractional integral and of the AB fractional derivatives of Riemann--Liouville type and Caputo type (ABR derivative and ABC derivative respectively) are as follows \cite{atangana-baleanu,baleanu-fernandez}.

\begin{definition}
\label{Def:AB}
Let $[a,b]$ be a real interval and $0<\nu<1$. The $\nu$th AB integral of a function $f\in L^1[a,b]$ is defined as follows, for any $x\in[a,b]$:
\begin{equation}
\label{ABint:def}
\prescript{AB}{a}I^{\nu}_xf(x)=\frac{1-\nu}{B(\nu)}f(x)+\frac{\nu}{B(\nu)}\prescript{RL}{a}I^{\nu}_xf(x).
\end{equation}
The $\nu$th ABR derivative (AB derivative of Riemann--Liouville type) of a function $f\in L^1[a,b]$ is defined as follows, for any $x\in[a,b]$:
\begin{equation}
\label{ABR:def}
\prescript{ABR}{a}D^{\nu}_xf(x)=\frac{B(\nu)}{1-\nu}\cdot\frac{\mathrm{d}}{\mathrm{d}x}\int_a^xE_{\nu}\left(\frac{-\nu}{1-\nu}(x-y)^{\nu}\right)f(y)\,\mathrm{d}y.
\end{equation}
Similarly, the $\nu$th ABC derivative (AB derivative of Caputo type) of a differentiable function $f$ on $[a,b]$ with $f'\in L^1[a,b]$ is defined as follows, for any $x\in[a,b]$:
\begin{equation}
\label{ABC:def}
\prescript{ABC}{a}D^{\nu}_xf(x)=\frac{B(\nu)}{1-\nu}\int_a^xE_{\nu}\left(\frac{-\nu}{1-\nu}(x-y)^{\nu}\right)f'(y)\,\mathrm{d}y.
\end{equation}
\end{definition}

Many properties of the AB model have been established in the literature. We present some of the most important results as follows.

\begin{proposition}[Series formula \cite{baleanu-fernandez}]
\label{Prop:ABseries}
For $0<\nu<1$ and for any $a$, $b$, $x$, $f$ as in Definition \ref{Def:AB}, we have the following locally uniformly convergent series expressions for the ABR and ABC derivatives:
\begin{align}
\label{ABR:series} \prescript{ABR}{a}D^{\nu}_xf(x)&=\frac{B(\nu)}{1-\nu}\sum_{n=0}^{\infty}\left(\frac{-\nu}{1-\nu}\right)^n\prescript{RL}{a}I^{n\nu}_xf(x); \\
\label{ABC:series} \prescript{ABC}{a}D^{\nu}_xf(x)&=\frac{B(\nu)}{1-\nu}\sum_{n=0}^{\infty}\left(\frac{-\nu}{1-\nu}\right)^n\prescript{RL}{a}I^{n\nu+1}_xf'(x).
\end{align}
\end{proposition}

\begin{proposition}[Inversion relations \cite{abdeljawad-baleanu,djida-atangana-area,baleanu-fernandez}]
\label{Prop:ABcompos}
For $0<\nu<1$ and for any $a$, $b$, $x$, $f$ as in Definition \ref{Def:AB}, we have the following composition relations between the AB derivatives and integrals:
\begin{align}
\label{ABcompos1} \prescript{ABR}{a}D^{\nu}_x\Big(\prescript{AB}{a}I^{\nu}_xf(x)\Big)&=f(x); \\
\label{ABcompos2} \prescript{AB}{a}I^{\nu}_x\Big(\prescript{ABR}{a}D^{\nu}_xf(x)\Big)&=f(x); \\
\label{ABcompos3} \prescript{AB}{a}I^{\nu}_x\Big(\prescript{ABC}{a}D^{\nu}_xf(x)\Big)&=f(x)-f(a).
\end{align}
However, the semigroup property for AB differintegrals does not hold in general. For $0<\mu,\nu<1$:
\begin{align}
\label{ABsemi1} \prescript{AB}{a}I^{\mu}_x\Big(\prescript{AB}{a}I^{\nu}_xf(x)\Big)&\neq\prescript{AB}{a}I^{\mu+\nu}_xf(x); \\
\label{ABsemi2} \prescript{ABR}{a}D^{\mu}_x\Big(\prescript{ABR}{a}D^{\nu}_xf(x)\Big)&\neq\prescript{ABR}{a}D^{\mu+\nu}_xf(x); \\
\label{ABsemi3} \prescript{ABC}{a}D^{\mu}_x\Big(\prescript{ABC}{a}D^{\nu}_xf(x)\Big)&\neq\prescript{ABC}{a}D^{\mu+\nu}_xf(x).
\end{align}
\end{proposition}

\begin{remark}
The results of equations \eqref{ABsemi1}--\eqref{ABsemi3} are especially important in the theory of the AB model. They demonstrate that the AB differintegral operators do not obey an index law, which fact has been discussed at more length in papers such as \cite{atangana2,atangana-gomez1,atangana-gomez2}.

It is also important to be aware, contrary to some claims seen in the literature, that the AB differintegral operators are commutative. This result is stated precisely in the following Proposition.
\end{remark}

\begin{proposition}[Commutativity relations \cite{baleanu-fernandez}]
\label{Prop:ABcommut}
For $0<\mu,\nu<1$ and for any $a$, $b$, $x$, $f$ as in Definition \ref{Def:AB}, we have the following composition relations between the ABR derivatives and AB integrals:
\begin{align}
\label{ABcommut1} \prescript{AB}{a}I^{\mu}_x\Big(\prescript{AB}{a}I^{\nu}_xf(x)\Big)&=\prescript{AB}{a}I^{\nu}_x\Big(\prescript{AB}{a}I^{\mu}_xf(x)\Big); \\
\label{ABcommut2} \prescript{ABR}{a}D^{\mu}_x\Big(\prescript{ABR}{a}D^{\nu}_xf(x)\Big)&=\prescript{ABR}{a}D^{\nu}_x\Big(\prescript{ABR}{a}D^{\mu}_xf(x)\Big); \\
\label{ABcommut3} \prescript{ABR}{a}D^{\nu}_x\Big(\prescript{AB}{a}I^{\mu}_xf(x)\Big)&=\prescript{AB}{a}I^{\mu}_x\Big(\prescript{ABR}{a}D^{\nu}_xf(x)\Big).
\end{align}
\end{proposition}

The \textbf{iterated AB model} was defined in \cite{fernandez-baleanu2} by first iterating the AB integral operator a whole number of times and then extending the series thus obtained to an infinite series representation for a fractional iteration of the AB integral operator.

\begin{definition}
\label{Def:IAB}
Let $[a,b]$ be a real interval, $0<\nu<1$, and $\mu\in\mathbb{R}$. The iterated AB differintegral to order $(\nu,\mu)$ of a function $f\in L^1[a,b]$ is defined as follows, for any $x\in[a,b]$:
\begin{align}
\label{IAB:RLdef} \prescript{IAB}{a}I^{\nu,\mu}_{x}f(x)&=\sum_{n=0}^{\infty}\frac{\binom{\mu}{n}(1-\nu)^{\mu-n}\nu^n}{B(\nu)^{\mu}}\prescript{RL}{a}I_{x}^{n\nu}f(x) \\
\label{IAB:2def} &=\left(\frac{1-\nu}{B(\nu)}\right)^{\mu}f(x)+\sum_{n=1}^{\infty}\frac{\binom{\mu}{n}(1-\nu)^{\mu-n}\nu^n}{B(\nu)^{\mu}\Gamma(n\nu)}\int_{a}^x(x-y)^{n\nu-1}f(y)\,\mathrm{d}y \\
\label{IAB:deltadef} &=\int_{a}^x\left(\frac{1-\nu}{B(\nu)}\right)^{\mu}f(y)\left[\delta(x-y)+\sum_{n=1}^{\infty}\frac{\binom{\mu}{n}(1-\nu)^{-n}\nu^n}{\Gamma(n\nu)}(x-y)^{n\nu-1}\right]\,\mathrm{d}y,
\end{align}
where $\delta$ is the Dirac delta function.
\end{definition}

In particular, when $\mu=1$ this yields the standard AB integral, when $\mu=m\in\mathbb{N}$ it is the $m$th iteration of the standard AB integral operator, and when $\mu=-1$ it is the ABR derivative. We mention here the most important property of the iterated AB model -- the semigroup property, its main advantage over the standard AB model -- but we invite readers to check \cite{fernandez-baleanu2} for more discussion of these operators and their properties.

\begin{proposition}[Semigroup property \cite{fernandez-baleanu2}]
\label{Prop:IABsemigroup}
For $\mu,\rho\in\mathbb{R}$ and for any $a,b,x,f,\nu$ as in Definition \ref{Def:IAB}, we have the following semigroup property in the second variable for iterated AB differintegrals:
\[\prescript{IAB}{a}I^{\nu,\mu}_{x}\Big(\prescript{IAB}{a}I^{\nu,\rho}_{x}f(x)\Big)=\prescript{IAB}{a}I^{\nu,\mu+\rho}_{x}f(x).\]
\end{proposition}

\section{The complex analysis approach} \label{Sec:main}

\subsection{AB derivatives} \label{Sec:main1}

First, we demonstrate that the definitions \eqref{ABR:def}--\eqref{ABC:def} for the ABR and ABC derivatives can be extended to complex values of the order of differentiation $\nu$ by using analytic continuation.

\begin{lemma} \label{Lem:ABderAC}
Let $c$ and $z$ be real numbers with $c<z$, let $f$ be a function as in Definition \ref{Def:AB}, and assume the multiplier function $B(\nu)$ is analytic. Then the integral formulae 
\begin{align}
\label{ABR:Cdef} \prescript{ABR}{c}D^{\nu}_zf(z)&=\frac{B(\nu)}{1-\nu}\cdot\frac{\mathrm{d}}{\mathrm{d}z}\int_c^zE_{\nu}\left(\frac{-\nu}{1-\nu}(z-y)^{\nu}\right)f(y)\,\mathrm{d}y, \\
\label{ABC:Cdef} \prescript{ABC}{c}D^{\nu}_zf(z)&=\frac{B(\nu)}{1-\nu}\int_c^zE_{\nu}\left(\frac{-\nu}{1-\nu}(z-y)^{\nu}\right)f'(y)\,\mathrm{d}y,
\end{align}
are well-defined for all $\nu\in\mathbb{C}$ with $\Real(\nu)>0$ and $\nu\neq1$, and they yield functions of $\nu$ which are analytic on the half-plane $\Real(\nu)>0$.
\end{lemma}

\begin{proof}
The function
\[E_{\nu}\left(\frac{-\nu}{1-\nu}(x-y)^{\nu}\right)=\sum_{n=0}^{\infty}\left(\frac{-\nu}{1-\nu}\right)^n\frac{(x-y)^{n\nu}}{\Gamma(n\nu+1)}\]
is clearly analytic as a function of $\nu$ on the domain $U\coloneqq\{\nu\in\mathbb{C}:\Real(\nu)>0,\;\nu\neq1\}$ when $c<y<x$, since the inverse gamma function is entire and the power functions are well-behaved at the singular point $y=x$ and the series is locally uniformly convergent. Thus, the integral formulae for the ABR and ABC derivatives yield analytic functions of $\nu\in U$. It remains to check the value $\nu=1$.

It was shown in \cite{caputo-fabrizio,atangana-baleanu} that as $\nu\rightarrow1$ the kernel function (being in this case an exponential function rather than a Mittag-Leffler function) approaches the Dirac delta distribution, and therefore the right-hand sides of both \eqref{ABR:Cdef} and \eqref{ABC:Cdef} become simply $f'(z)$. So the singularity at $\nu=1$ is removable, and the result follows.
\end{proof}

Lemma \ref{Lem:ABderAC} provides a natural \textbf{analytic continuation} of the ABR and ABC derivatives to general complex numbers $\nu$. We can also propose an analogue of the Cauchy formula \eqref{RLdef:Cauchy} which gives an alternative formulation for these generalised AB derivatives, and which once again is valid on the left half-plane ($\Real(\nu)<0$) as well as the right ($\Real(\nu)>0$).

\begin{lemma} \label{Lem:Cseries}
The series formulae for the ABR and ABC derivatives, given by Proposition \ref{Prop:ABseries}, are still valid for the extended definitions given by \eqref{ABR:Cdef}--\eqref{ABC:Cdef} for all $\nu$ with $\Real(\nu)>0$.
\end{lemma}

\begin{proof}
The proof is the same as in \cite{baleanu-fernandez}, utilising the Taylor series for the Mittag-Leffler function which is valid for all values of $\nu$.
\end{proof}

\begin{lemma}
\label{Lem:ABseriesCauchy}
Let $c$ and $z$ be real numbers with $c<z$, let $\nu\in\mathbb{C}$ with $\Real(\nu)>0$ and $\nu\not\in\mathbb{Q}$, and let $f$ be a complex function which is analytic on an open neighbourhood of the straight line-segment $[c,z]$. Then the ABR and ABC derivatives of $f$ can be written in the following form:
\begin{align}
\label{ABR:seriesCauchy} \prescript{ABR}{c}D^{\nu}_zf(z)&=\frac{B(\nu)}{2\pi i(1-\nu)}\cdot\frac{\mathrm{d}}{\mathrm{d}z}\sum_{n=0}^{\infty}\int_H\Gamma(-n\nu)\left(\frac{-\nu}{1-\nu}(w-z)\right)^{n\nu}f(w)\,\mathrm{d}w, \\
\label{ABC:seriesCauchy} \prescript{ABC}{c}D^{\nu}_zf(z)&=\frac{B(\nu)}{2\pi i(1-\nu)}\sum_{n=0}^{\infty}\int_H\Gamma(-n\nu)\left(\frac{-\nu}{1-\nu}(w-z)\right)^{n\nu}f'(w)\,\mathrm{d}w,
\end{align}
where the series are locally uniformly convergent and the Hankel contour $H$ is defined as above by equations \eqref{Hankel1}--\eqref{Hankel3}.
\end{lemma}

\begin{proof}
By Lemma \ref{Lem:Cseries}, we can use the expressions for the AB derivatives as infinite series of RL integrals and then rewrite each RL integral in the Cauchy form:
\begin{align*}
\prescript{ABR}{c}D^{\nu}_zf(z)&=\frac{B(\nu)}{1-\nu}\cdot\frac{\mathrm{d}}{\mathrm{d}z}\sum_{n=0}^{\infty}\left(\frac{-\nu}{1-\nu}\right)^n\prescript{RL}{c}I^{n\nu+1}_zf(z) \\
&=\frac{B(\nu)}{1-\nu}\cdot\frac{\mathrm{d}}{\mathrm{d}z}\sum_{n=0}^{\infty}\left(\frac{-\nu}{1-\nu}\right)^n\prescript{\mathbb{C}}{c}I^{n\nu+1}_zf(z) \\
&=\frac{B(\nu)}{1-\nu}\cdot\frac{\mathrm{d}}{\mathrm{d}z}\sum_{n=0}^{\infty}\left(\frac{-\nu}{1-\nu}\right)^n\frac{\Gamma(-n\nu-1+1)}{2\pi i}\int_H(w-z)^{n\nu+1-1}f(w)\,\mathrm{d}w \\
&=\frac{B(\nu)}{2\pi i(1-\nu)}\cdot\frac{\mathrm{d}}{\mathrm{d}z}\sum_{n=0}^{\infty}\int_H\Gamma(-n\nu)\left(\frac{-\nu}{1-\nu}(w-z)\right)^{n\nu}f(w)\,\mathrm{d}w, \\
\prescript{ABC}{c}D^{\nu}_zf(z)&=\frac{B(\nu)}{1-\nu}\sum_{n=0}^{\infty}\left(\frac{-\nu}{1-\nu}\right)^n\prescript{RL}{c}I^{n\nu+1}_zf'(z) \\
&=\frac{B(\nu)}{1-\nu}\sum_{n=0}^{\infty}\left(\frac{-\nu}{1-\nu}\right)^n\prescript{\mathbb{C}}{c}I^{n\nu+1}_zf'(z) \\
&=\frac{B(\nu)}{1-\nu}\sum_{n=0}^{\infty}\left(\frac{-\nu}{1-\nu}\right)^n\frac{\Gamma(-n\nu-1+1)}{2\pi i}\int_H(w-z)^{n\nu+1-1}f'(w)\,\mathrm{d}w \\
&=\frac{B(\nu)}{2\pi i(1-\nu)}\sum_{n=0}^{\infty}\int_H\Gamma(-n\nu)\left(\frac{-\nu}{1-\nu}(w-z)\right)^{n\nu}f'(w)\,\mathrm{d}w,
\end{align*}
exactly as required. We require $\nu\not\in\mathbb{Q}$ so that $n\nu+1$ is never in $\mathbb{N}$ and the Riemann--Liouville integral $\prescript{RL}{c}I^{n\nu+1}_z$ can always be rewritten in Cauchy form using \eqref{RLdef:Cauchy}.
\end{proof}

Note that once again \cite{baleanu-fernandez,fernandez-baleanu,fernandez-baleanu-srivastava} the series connection between the AB and RL models has turned out to be useful in providing quick proofs of results concerning the AB model.

Formally, we can now define a modified Mittag-Leffler function $E^{\nu}$ by
\begin{equation}
\label{MML}
E^{\nu}(x)=\sum_{n=0}^{\infty}\Gamma(-n\nu)x^n,
\end{equation}
and, without justifying any convergence properties, rewrite the expressions \eqref{ABR:seriesCauchy}--\eqref{ABC:seriesCauchy} as follows:
\begin{align}
\label{ABR:Cauchy} \prescript{ABR}{c}D^{\nu}_zf(z)&=\frac{B(\nu)}{2\pi i(1-\nu)}\cdot\frac{\mathrm{d}}{\mathrm{d}z}\int_HE^{\nu}\left(\frac{-\nu}{1-\nu}(w-z)^{\nu}\right)f(w)\,\mathrm{d}w, \\
\label{ABC:Cauchy} \prescript{ABC}{c}D^{\nu}_zf(z)&=\frac{B(\nu)}{2\pi i(1-\nu)}\int_H E^{\nu}\left(\frac{-\nu}{1-\nu}(w-z)^{\nu}\right)f'(w)\,\mathrm{d}w.
\end{align}
These expressions can be justified rigorously provided the series from \eqref{MML} has sufficient convergence properties. We establish such properties in the following lemma.

\begin{lemma} \label{Lem:MMLconv}
The series \eqref{MML} for the modified Mittag-Leffler function converges provided that $\nu\in\mathbb{C}\backslash\mathbb{R}$, locally uniformly in both $x$ and $\nu$.
\end{lemma}

\begin{proof}
Using the definition \eqref{MML} and the reflection formula for the gamma function, we have
\begin{align*}
E^{\nu}(x)&=\sum_{n=0}^{\infty}\Gamma(-n\nu)x^n=\sum_{n=0}^{\infty}\frac{\pi}{\sin(-n\pi\nu)}\cdot\frac{x^n}{\Gamma(1+n\nu)} \\
&=2\pi i\sum_{n=0}^{\infty}\frac{1}{\exp(-in\pi\nu)-\exp(in\pi\nu)}\cdot\frac{x^n}{\Gamma(1+n\nu)}.
\end{align*}
The last series is a modification of the standard Mittag-Leffler function by a factor which, for fixed $\nu\in\mathbb{C}\backslash\mathbb{R}$ and depending on the sign of $\Imag(\nu)$, tends asymptotically to either $\exp(in\pi\nu)$ or $\exp(-in\pi\nu)$ as $n\rightarrow\infty$. In both cases, there is exponential decay, with the $n$th term of the series being asymptotically approximated by
\[\frac{2\pi i}{\Gamma(1+n\nu)}\left(\frac{x}{\exp(\pm i\pi\nu)}\right)^n,\]
and the series converges locally uniformly in each of $x\in\mathbb{C}$ and $\nu\in\mathbb{C}\backslash\mathbb{R}$.
\end{proof}

Combining the results of Lemma \ref{Lem:ABseriesCauchy} and Lemma \ref{Lem:MMLconv}, we find that the formulae in \eqref{ABR:Cauchy}--\eqref{ABC:Cauchy} are valid expressions for the ABR and ABC derivatives whenever $\Real(\nu)>0$ and $\nu\not\in\mathbb{R}$. In other words, we have the following theorem, which we shall now prove by another method analogous to that used to prove the Cauchy formula \eqref{RLdef:Cauchy} for RL fractional integrals.

\begin{theorem} \label{Thm:ABCauchy}
Let $c$ and $z$ be real numbers with $c<z$, let $\nu\in\mathbb{C}$ with $\Real(\nu)>0$ and $\nu\not\in\mathbb{R}$, and let $f$ be a complex function which is analytic on an open neighbourhood of the straight line-segment $[c,z]$. Then the formulae \eqref{ABR:Cdef} and \eqref{ABC:Cdef} for the ABR and ABC derivatives of $f$ can be rewritten in the following form:
\begin{align*}
\prescript{ABR}{c}D^{\nu}_zf(z)&=\frac{B(\nu)}{2\pi i(1-\nu)}\cdot\frac{\mathrm{d}}{\mathrm{d}z}\int_HE^{\nu}\left(\frac{-\nu}{1-\nu}(w-z)^{\nu}\right)f(w)\,\mathrm{d}w, \\
\prescript{ABC}{c}D^{\nu}_zf(z)&=\frac{B(\nu)}{2\pi i(1-\nu)}\int_H E^{\nu}\left(\frac{-\nu}{1-\nu}(w-z)^{\nu}\right)f'(w)\,\mathrm{d}w,
\end{align*}
where $E^{\nu}$ denotes the modified Mittag-Leffler function defined by \eqref{MML} and $H$ denotes the Hankel contour defined by \eqref{Hankel1}--\eqref{Hankel3}.
\end{theorem}

\begin{proof}
It will be enough to show that the integral transform
\begin{equation}
\label{ABtransform}
\mathcal{E_{\nu}}[f](z)=\int_c^zE_{\nu}\left(\frac{-\nu}{1-\nu}(z-y)^{\nu}\right)f(y)\,\mathrm{d}y
\end{equation}
is equivalent to the complex integral transform
\begin{equation}
\label{ABtransform:complex}
\mathcal{E^{\nu}}[f](z)=\frac{1}{2\pi i}\int_HE^{\nu}\left(\frac{-\nu}{1-\nu}(w-z)^{\nu}\right)f(w)\,\mathrm{d}w,
\end{equation}
because we can then combine these transforms with the differentiation operation to get the ABR and ABC derivatives in the form required.

We start from the complex form \eqref{ABtransform:complex} and recall the definition \eqref{Hankel1}--\eqref{Hankel3} of the Hankel contour.

For $w$ on the contour $H_1$ defined by \eqref{Hankel1}, we have \[w-z=re^{-i\pi}(z-c)=e^{-i\pi}(z-w),\] where in this case both $z-w$ and $z-c$ are positive reals. Similarly, for $w$ on the contour $H_3$ defined by \eqref{Hankel3}, we have \[w-z=re^{i\pi}(z-c)=e^{i\pi}(z-w).\] And for $w$ on the contour $H_2$ defined by \eqref{Hankel2}, we have \[w-z=\epsilon e^{i\theta}(z-c)=\epsilon' e^{i\theta'}\] where $\epsilon'\coloneqq\epsilon|z-c|$ is a small positive real independent of $w$. Therefore, the integral appearing in \eqref{ABtransform:complex} can be written as:
\begin{align*}
2\pi i\mathcal{E^{\nu}}[f](z)&=\int_{H_1}E^{\nu}\left(\frac{-\nu}{1-\nu}(w-z)^{\nu}\right)f(w)\,\mathrm{d}w+\int_{H_3}E^{\nu}\left(\frac{-\nu}{1-\nu}(w-z)^{\nu}\right)f(w)\,\mathrm{d}w \\
&\hspace{3cm}+\int_{H_2}E^{\nu}\left(\frac{-\nu}{1-\nu}(w-z)^{\nu}\right)f(w)\,\mathrm{d}w \\
&=\int_c^zE^{\nu}\left(\frac{-\nu}{1-\nu}e^{-i\pi\nu}(z-w)^{\nu}\right)f(w)\,\mathrm{d}w+\int_z^cE^{\nu}\left(\frac{-\nu}{1-\nu}e^{i\pi\nu}(z-w)^{\nu}\right)f(w)\,\mathrm{d}w \\
&\hspace{1.5cm}+\int_{-\pi}^{\pi}E^{\nu}\left(\frac{-\nu}{1-\nu}\epsilon^{\nu}e^{i\theta\nu}(z-c)^{\nu}\right)f(z+\epsilon e^{i\theta}(z-c))i\epsilon e^{i\theta}(z-c)\,\mathrm{d}\theta.
\end{align*}
The third integral (over $H_2$) has an integrand which can be written as
\[E^{\nu}\Big(O(\epsilon^{\nu})\Big)\big[f(z)+O(\epsilon)\big]O(\epsilon)=\big[1+O(\epsilon^{\nu})\big]\big[f(z)+O(\epsilon)\big]O(\epsilon)=O(\epsilon),\]
where the big-$O$ approximation is taken as $\epsilon$ approaches zero, and where we have used the fact that $\Real(\nu)>0$. Therefore, the integral over $H_2$ approaches zero as $\epsilon\rightarrow0$, and it can be ignored in the subsequent calculation. We are left with the following expression for the integral from \eqref{ABtransform:complex}:
\begin{equation}
\label{ABCauchy:step}
2\pi i\mathcal{E}^{\nu}[f](z)=\int_c^z\left[E^{\nu}\left(\tfrac{-\nu}{1-\nu}e^{-i\pi\nu}(z-w)^{\nu}\right)-E^{\nu}\left(\tfrac{-\nu}{1-\nu}e^{i\pi\nu}(z-w)^{\nu}\right)\right]f(w)\,\mathrm{d}w.
\end{equation}
Now, recalling the definition of the $E^{\nu}$ function, we rewrite the expression in square brackets on the right-hand side of \eqref{ABCauchy:step} as follows:
\begin{align*}
&\hspace{-1cm}E^{\nu}\left(\tfrac{-\nu}{1-\nu}e^{-i\pi\nu}(z-w)^{\nu}\right)-E^{\nu}\left(\tfrac{-\nu}{1-\nu}e^{i\pi\nu}(z-w)^{\nu}\right) \\
&=\sum_{n=0}^{\infty}\Gamma(-n\nu)\left(\frac{-\nu}{1-\nu}\right)^ne^{-in\pi\nu}(z-w)^{n\nu}-\sum_{n=0}^{\infty}\Gamma(-n\nu)\left(\frac{-\nu}{1-\nu}\right)^ne^{in\pi\nu}(z-w)^{n\nu} \\
&=\sum_{n=0}^{\infty}\Gamma(-n\nu)\left(\frac{-\nu}{1-\nu}\right)^n(z-w)^{n\nu}\left[e^{-in\pi\nu}-e^{in\pi\nu}\right] \\
&=\sum_{n=0}^{\infty}\left[\frac{\pi}{\Gamma(1+n\nu)\sin(\pi(-n\nu))}\right]\left(\frac{-\nu}{1-\nu}\right)^n(z-w)^{n\nu}\left[-2i\sin(n\pi\nu)\right] \\
&=\sum_{n=0}^{\infty}\frac{2\pi i}{\Gamma(1+n\nu)}\left(\frac{-\nu}{1-\nu}\right)^n(z-w)^{n\nu}=2\pi iE_{\nu}\left(\frac{-\nu}{1-\nu}(z-w)^{\nu}\right),
\end{align*}
where we have used the reflection formula for $\Gamma$, and where $E_{\nu}$ is the usual Mittag-Leffler function. Thus, \eqref{ABCauchy:step} becomes
\[2\pi i\mathcal{E^{\nu}}[f](z)=\int_c^z\left[2\pi iE_{\nu}\left(\frac{-\nu}{1-\nu}(z-w)^{\nu}\right)\right]f(w)\,\mathrm{d}w=2\pi i\mathcal{E_{\nu}}[f](z),\]
and we have proved the equivalence of \eqref{ABtransform} and \eqref{ABtransform:complex} as required.
\end{proof}

Now we are finally in a position to define analytic continuations of the ABR and ABC derivatives to the whole complex plane for $\nu$. It follows from Lemma \ref{Lem:MMLconv} that the formulae \eqref{ABR:Cauchy}--\eqref{ABC:Cauchy} established in Theorem \ref{Thm:ABCauchy} are in fact analytic functions of $\nu\in\mathbb{C}\backslash\mathbb{R}$, while we showed in Lemma \ref{Lem:ABderAC} that the equivalent formulae \eqref{ABR:Cdef}--\eqref{ABC:Cdef} are analytic functions of $\nu$ in the right half-plane $\Real(\nu)>0$. Analyticity in $z$ is preserved throughout if the function $f$ itself is analytic. Therefore, we propose the following definition.

\begin{definition} \label{Def:EABder}
Let $c$ be a fixed complex number and $f$ be a complex function which is analytic on an open star-domain $R$ centred at $c$.

The \textbf{extended ABR derivative} $\prescript{ABR}{c}D^{\nu}_zf(z)$ is defined for any $\nu\in\mathbb{C}\backslash\mathbb{R}^-$ and any $z\in R\backslash\{c\}$ by:
\begin{alignat}{2}
\label{EABR:1} \prescript{ABR}{c}D^{\nu}_zf(z)&=\frac{B(\nu)}{1-\nu}\cdot\frac{\mathrm{d}}{\mathrm{d}z}\int_c^zE_{\nu}\left(\frac{-\nu}{1-\nu}(z-y)^{\nu}\right)f(y)\,\mathrm{d}y,\quad&&\quad\Real(\nu)>0; \\
\label{EABR:2} \prescript{ABR}{c}D^{\nu}_zf(z)&=\frac{B(\nu)}{2\pi i(1-\nu)}\cdot\frac{\mathrm{d}}{\mathrm{d}z}\int_HE^{\nu}\left(\frac{-\nu}{1-\nu}(w-z)^{\nu}\right)f(w)\,\mathrm{d}w,\quad&&\quad\nu\in\mathbb{C}\backslash\mathbb{R}.
\end{alignat}

The \textbf{extended ABC derivative} $\prescript{ABC}{c}D^{\nu}_zf(z)$ is defined for any $\nu\in\mathbb{C}\backslash\mathbb{R}^-$ and any $z\in R\backslash\{c\}$ by:
\begin{alignat}{2}
\label{EABC:1} \prescript{ABC}{c}D^{\nu}_zf(z)&=\frac{B(\nu)}{1-\nu}\int_c^zE_{\nu}\left(\frac{-\nu}{1-\nu}(z-y)^{\nu}\right)f'(y)\,\mathrm{d}y,\quad&&\quad\Real(\nu)>0; \\
\label{EABC:2} \prescript{ABC}{c}D^{\nu}_zf(z)&=\frac{B(\nu)}{2\pi i(1-\nu)}\int_HE^{\nu}\left(\frac{-\nu}{1-\nu}(w-z)^{\nu}\right)f'(w)\,\mathrm{d}w,\quad&&\quad\nu\in\mathbb{C}\backslash\mathbb{R}.
\end{alignat}
\end{definition}

\begin{proposition}
\label{Prop:EAB:AC}
The extended ABR and ABC derivatives proposed in Definition \ref{Def:EABder} are:
\begin{enumerate}
\item analytic functions of both $z\in R\backslash\{c\}$ and $\nu\in\mathbb{C}\backslash\mathbb{R}^-_0$, provided $f$ and $B$ are analytic;
\item identical to the original formulae \eqref{ABR:def}--\eqref{ABC:def} in the case when $0<\nu<1$ and $c<z$ in $\mathbb{R}$.
\end{enumerate}
Therefore, they provide \textbf{analytic continuations} of the original ABR and ABC derivatives to complex values of $z$ and $\nu$.
\end{proposition}

\begin{proof}
We showed in Lemma \ref{Lem:ABderAC} that the formulae \eqref{EABR:1} and \eqref{EABC:1} are valid extensions of the original ABR and ABC derivatives from $0<\nu<1$ to $\Real(\nu)>0$.

Meanwhile, it follows from Lemma \ref{Lem:MMLconv} that the formulae \eqref{EABR:2} and \eqref{EABC:2} are in fact analytic functions of $\nu\in\mathbb{C}\backslash\mathbb{R}$, and Theorem \ref{Thm:ABCauchy} tells us that these formulae agree with \eqref{EABR:1} and \eqref{EABC:1} respectively on the region $\Real(\nu)>0,\nu\not\in\mathbb{R}$.

The above results were proved for real $z$, but analyticity in $z\in R\backslash\{c\}$ is preserved throughout if the function $f$ itself is analytic.
\end{proof}

\subsection{AB integrals and the iterated AB model} \label{Sec:main2}

The analytic continuation of the AB integral defined by \eqref{ABint:def} is much simpler to manage than that of the AB derivatives, since this time we can simply consider RL integrals rather than Mittag-Leffler series.

\begin{definition} \label{Def:EABint}
Let $c$ be a fixed complex number and $f$ be a complex function which is analytic on an open star-domain $R$ centred at $c$. The \textbf{extended AB integral} $\prescript{AB}{c}I^{\nu}_zf(z)$ is defined for any $\nu\in\mathbb{C}$ and any $z\in R\backslash\{c\}$ by:
\begin{equation}
\label{EABint}
\prescript{AB}{c}I^{\nu}_zf(z)=\frac{1-\nu}{B(\nu)}f(z)+\frac{\nu}{B(\nu)}\prescript{RL}{c}I^{\nu}_zf(z).
\end{equation}
\end{definition}

\begin{proposition}
\label{Prop:EABI:AC}
The extended AB integral proposed in Definition \ref{Def:EABint} is:
\begin{enumerate}
\item an analytic function of both $z\in R\backslash\{c\}$ and $\nu\in\mathbb{C}$, provided $f$ and $B$ are analytic and $B$ is nonzero;
\item identical to the original formula \eqref{ABint:def} in the case when $0<\nu<1$ and $c<z$ in $\mathbb{R}$.
\end{enumerate}
Therefore, it provides the \textbf{analytic continuation} of the original AB integral to complex values of $z$ and $\nu$.
\end{proposition}

\begin{proof}
This follows from the analyticity of the original Riemann--Liouville differintegral $\prescript{RL}{c}I^{\nu}_zf(z)$ in $z$ and $\nu$, which is well known.
\end{proof}

\begin{theorem} \label{Thm:ABintCauchy}
With all notation as in Definition \ref{Def:EABint}, an alternative form for the extended AB integral is as follows:
\begin{equation}
\label{ABint:Cauchy}
\prescript{AB}{c}I^{\nu}_zf(z)=\frac{1}{2\pi iB(\nu)}\int_H \left(\frac{1-\nu}{w-z}+\frac{\nu\Gamma(1-\nu)}{(w-z)^{1-\nu}}\right)f(w)\,\mathrm{d}z,\quad\quad z\in R\backslash\{c\},\nu\in\mathbb{C}\backslash\mathbb{N}.
\end{equation}
where $H$ is the Hankel contour defined in \eqref{Hankel1}--\eqref{Hankel3}.
\end{theorem}

\begin{proof}
We use the Cauchy form \eqref{RLdef:Cauchy} for the Riemann--Liouville differintegral and the standard Cauchy integral formula for analytic functions:
\begin{align*}
f(z)&=\frac{1}{2\pi i}\int_{|w-z|=\epsilon}\frac{f(w)}{w-z}\,\mathrm{d}w=\frac{1}{2\pi i}\int_H\frac{f(w)}{w-z}\,\mathrm{d}w, \\
\prescript{RL}{c}I^{\nu}_zf(z)&=\prescript{\mathbb{C}}{c}D^{-\nu}_zf(z)=\frac{\Gamma(-\nu+1)}{2\pi i}\int_H(w-
z)^{\nu-1}f(w)\,\mathrm{d}w.
\end{align*}
By taking a linear combination of these two integrals, the result follows.
\end{proof}

\begin{remark} \label{Rem:ABnodelta}
For the AB integral operator, an advantage of the Cauchy formula \eqref{ABint:Cauchy} over the original formula \eqref{EABint} is that the former enables us to express the operator in terms of a single kernel function. In the original definition \eqref{ABint:def} of AB integrals, it was necessary to take a linear combination of an integral expression with the function $f$ itself. In the extended definition \eqref{ABint:Cauchy}, the AB integral is expressed as the convolution of $f$ with a single simple kernel function.

The reason for this change is that, when we are doing complex integration, any function $f(z)$ can be expressed as a convolution of itself with the reciprocal function $\frac{1}{z}$. By contrast, in real analysis the only way to express a function as a convolution of itself is to use the Dirac delta function, which is not a function and which makes the analysis more complicated.

Thus, we already see an advantage of the complex-analytic approach in terms of the simplicity and elegance of the definitions.
\end{remark}

We now consider the iterated AB model which was defined in equations \eqref{IAB:RLdef}--\eqref{IAB:deltadef} above. For these operators, the series formula using RL integrals (namely \eqref{IAB:RLdef}, analogous to \eqref{ABR:series}--\eqref{ABC:series} for the AB derivatives) is in a nice simple form, but the integral formula using a complicated kernel (namely \eqref{IAB:deltadef}, analogous to \eqref{ABR:def}--\eqref{ABC:def} for the AB derivatives) is much trickier and the kernel is not even a function. Thus, it would be useful to have a different formulation of the iterated AB differintegral, using an integral with a nicer kernel function. As we shall see, such a formulation can be given by using complex analysis and Cauchy-type integrals as above. This is a major advantage of the complex approach proposed in this paper: it gives a better and more elegant integral representation of the iterated AB differintegral.

\begin{lemma}
\label{Lem:IABintAC}
Let $c$ and $z$ be real numbers with $c<z$, let $f$ be an $L^1$ function on an interval containing $[c,z]$, and assume the multiplier function $B(\nu)$ is analytic and nonzero. Then the integral formula
\begin{equation}
\label{IAB:Cdef}
\prescript{IAB}{c}I^{\nu,\mu}_{z}f(z)=\left(\frac{1-\nu}{B(\nu)}\right)^{\mu}f(z)+\int_{c}^z\sum_{n=1}^{\infty}\frac{\binom{\mu}{n}(1-\nu)^{\mu-n}\nu^n}{B(\nu)^{\mu}\Gamma(n\nu)}(z-w)^{n\nu-1}f(w)\,\mathrm{d}w
\end{equation}
is well-defined for all $\mu,\nu\in\mathbb{C}$ with $\Real(\nu)>0$ and $\nu\neq1$, and it yields a function which is entire in $\mu$ and analytic  in $\nu$ on the domain $\{\nu\in\mathbb{C}:\Real(\nu)>0,\nu\neq1\}$.
\end{lemma}

\begin{proof}
We start from the formula \eqref{IAB:2def} for the iterated AB differintegral, and note that this is identical to the series in the statement of the lemma. We just need to check convergence and analyticity, for which we rewrite the integral expression as follows:
\begin{align*}
\prescript{IAB}{c}I^{\nu,\mu}_{z}f(z)-\left(\frac{1-\nu}{B(\nu)}\right)^{\mu}f(z)&=\int_{c}^z\sum_{n=1}^{\infty}\frac{\binom{\mu}{n}(1-\nu)^{\mu-n}\nu^n}{B(\nu)^{\mu}\Gamma(n\nu)}(z-w)^{n\nu-1}f(w)\,\mathrm{d}w \\
&\hspace{-2cm}=\left(\frac{1-\nu}{B(\nu)}\right)^{\mu}\int_{c}^z\sum_{n=1}^{\infty}\frac{\Gamma(\mu+1)(1-\nu)^{-n}\nu^n}{\Gamma(n+1)\Gamma(\mu-n+1)\Gamma(n\nu)}(z-w)^{n\nu}\frac{f(w)}{z-w}\,\mathrm{d}w \\
&\hspace{-2cm}=\left(\frac{1-\nu}{B(\nu)}\right)^{\mu}\int_{c}^z\sum_{n=1}^{\infty}\frac{(\mu)(\mu-1)\dots(\mu-n+1)}{n!\Gamma(n\nu)}\left(\frac{\nu}{1-\nu}(z-w)^{\nu}\right)^n\frac{f(w)}{z-w}\,\mathrm{d}w
\end{align*}
The singularity at $w=z$ is integrable provided that the exponent $n\nu-1$ is always greater than $-1$ in real part, hence we require $\Real(\nu)>0$. Meanwhile, the series in $n$ is locally uniformly convergent for all $z-w\in\mathbb{C}$. Analyticity in $\mu$ and $\nu$ is clear from the formula, with the only potential singularity being at $\nu=1$. The result follows.
\end{proof}

\begin{remark}
\label{Rem:IABserAC}
We now know that the series \eqref{IAB:RLdef} of Riemann--Liouville integrals also converges locally uniformly to a function which is entire in $\mu$ and analytic in $\nu$ on the domain $\{\nu\in\mathbb{C}:\Real(\nu)>0,\nu\neq1\}$. This is because, in the proof of Lemma \ref{Lem:IABintAC}, we saw that the series in the integrand was locally uniformly convergent, so the summation and integration may be swapped to yield a formula equivalent to \eqref{IAB:RLdef}.

Clearly, too, the formula \eqref{IAB:deltadef} is also well-defined and analytic for these extended complex values of $\mu$ and $\nu$, since it is exactly equivalent to the formula \eqref{IAB:2def} in all cases.
\end{remark}

\begin{lemma} \label{Lem:IABseriesCauchy}
Let $c$ and $z$ be real numbers with $c<z$, let $\nu\in\mathbb{C}$ with $\Real(\nu)>0$ and $\nu\not\in\mathbb{Q}$, and let $f$ be a complex function which is analytic on an open neighbourhood of the straight line-segment $[c,z]$. Then the iterated AB differintegral of $f$ can be written in the following form:
\begin{equation}
\label{IAB:seriesCauchy} 
\prescript{IAB}{c}I^{\nu,\mu}_{z}f(z)=\frac{1}{2\pi i}\left(\frac{1-\nu}{B(\nu)}\right)^{\mu}\sum_{n=0}^{\infty}\int_H\frac{\Gamma(\mu+1)\Gamma(1-n\nu)}{\Gamma(n+1)\Gamma(\mu-n+1)}\left(\frac{\nu}{1-\nu}(w-z)^{\nu}\right)^n\frac{f(w)}{w-z}\,\mathrm{d}w
\end{equation}
where the series is locally uniformly convergent and the Hankel contour $H$ is defined as above by equations \eqref{Hankel1}--\eqref{Hankel3}.
\end{lemma}

\begin{proof}
We use the series formula \eqref{IAB:RLdef}, since by Remark \ref{Rem:IABserAC} we know that this formula is still valid for complex $\mu$ and $\nu$ in the stated domains. Rewriting each RL integral from this series in the Cauchy form \eqref{RLdef:Cauchy}, we find:
\begin{align*}
\prescript{IAB}{c}I^{\nu,\mu}_{z}f(z)&=\sum_{n=0}^{\infty}\frac{\binom{\mu}{n}(1-\nu)^{\mu-n}\nu^n}{B(\nu)^{\mu}}\prescript{RL}{c}I_{z}^{n\nu}f(z) \\
&=\sum_{n=0}^{\infty}\frac{\binom{\mu}{n}(1-\nu)^{\mu-n}\nu^n}{B(\nu)^{\mu}}\prescript{\mathbb{C}}{c}I_{z}^{n\nu}f(z) \\
&=\sum_{n=0}^{\infty}\frac{\binom{\mu}{n}(1-\nu)^{\mu-n}\nu^n}{B(\nu)^{\mu}}\frac{\Gamma(-n\nu+1)}{2\pi i}\int_H(w-z)^{n\nu-1}f(w)\,\mathrm{d}w \\
&=\left(\frac{1-\nu}{B(\nu)}\right)^{\mu}\sum_{n=0}^{\infty}\int_H\binom{\mu}{n}\frac{\Gamma(-n\nu+1)}{2\pi i}\left(\frac{\nu}{1-\nu}\right)^n(w-z)^{n\nu-1}f(w)\,\mathrm{d}w,
\end{align*}
which yields the stated formula when we recall the definition of the binomial coefficient $\binom{\mu}{n}$. We require $\nu\not\in\mathbb{Q}$ so that $n\nu$ is never in $\mathbb{N}$ and the Riemann--Liouville integral $\prescript{RL}{c}I^{n\nu}_z$ can always be rewritten in Cauchy form using \eqref{RLdef:Cauchy}.
\end{proof}

Now formally we can define a function $E^{(\mu,\nu)}(x)$ by the series
\begin{equation}
\label{MML2}
E^{\mu,\nu}(x)=\sum_{n=0}^{\infty}\frac{\Gamma(\mu+1)\Gamma(1-n\nu)}{\Gamma(n+1)\Gamma(\mu-n+1)}x^n=\sum_{n=0}^{\infty}\binom{\mu}{n}\Gamma(-n\nu+1)x^n,
\end{equation}
and, without justifying any convergence properties, rewrite the series \eqref{IAB:seriesCauchy} as follows:
\[\prescript{IAB}{c}I^{\nu,\mu}_{z}f(z)=\frac{1}{2\pi i}\left(\frac{1-\nu}{B(\nu)}\right)^{\mu}\int_HE^{(\mu,\nu)}\left(\frac{\nu}{1-\nu}(w-z)^{\nu}\right)\frac{f(w)}{w-z}\,\mathrm{d}w.\]
As before, this expression can be justified rigorously provided that the series \eqref{MML2} has sufficient convergence properties.

\begin{lemma} \label{Lem:MML2conv}
The series \eqref{MML2} converges provided that $\nu\in\mathbb{C}\backslash\mathbb{R}$, locally uniformly in $x$, $\mu$, and $\nu$.
\end{lemma}

\begin{proof}
This can be proved in the same way as Lemma \ref{Lem:MMLconv}. Indeed, the series \eqref{MML2} is just a modification of the series \eqref{MML} by a factor of $\binom{\mu}{n}$ in the $n$th term, which does not affect convergence \cite[\S IV.3]{miller-ross}.
\end{proof}

\begin{theorem} \label{Thm:IABCauchy}
Let $c$ and $z$ be real numbers with $c<z$, let $\mu,\nu\in\mathbb{C}$ with $\Real(\nu)>0$ and $\nu\not\in\mathbb{R}$, and let $f$ be a complex function which is analytic on an open neighbourhood of the straight line-segment $[c,z]$. Then the iterated AB differintegral of $f$ can be rewritten in the following form:
\begin{align*}
\prescript{IAB}{c}I^{\nu,\mu}_{z}f(z)&=\frac{1}{2\pi i}\left(\frac{1-\nu}{B(\nu)}\right)^{\mu}\int_HE^{(\mu,\nu)}\left(\frac{\nu}{1-\nu}(w-z)^{\nu}\right)\frac{f(w)}{w-z}\,\mathrm{d}w \\
&=\frac{1}{2\pi i}\left(\frac{1-\nu}{B(\nu)}\right)^{\mu}\int_H\sum_{n=0}^{\infty}\frac{\Gamma(\mu+1)\Gamma(1-n\nu)}{\Gamma(n+1)\Gamma(\mu-n+1)}\left(\frac{\nu}{1-\nu}(w-z)^{\nu}\right)^n\frac{f(w)}{w-z}\,\mathrm{d}w,
\end{align*}
where $E^{(\mu,\nu)}$ denotes the modified double Mittag-Leffler function defined by \eqref{MML} and $H$ denotes the Hankel contour defined by \eqref{Hankel1}--\eqref{Hankel3}.
\end{theorem}

\begin{proof}
On the one hand, we can use the series \eqref{IAB:seriesCauchy} from Lemma \ref{Lem:IABseriesCauchy}. The convergence property given by Lemma \ref{Lem:MML2conv} enables us to swap the summation and integration in this expression, which yields the desired result.

Alternatively, on the other hand, we can use the series \eqref{IAB:Cdef} from Lemma \ref{Lem:IABintAC}. The strategy of complex contour integration, which we used in the proof of Theorem \ref{Thm:ABCauchy}, enables us to rewrite each term in the series from \eqref{IAB:Cdef} as a Hankel contour integral, and thus we obtain a series $\sum_{n=1}^{\infty}(\dots)$. For the required $n=0$ term, we observe that
\[\left(\frac{1-\nu}{B(\nu)}\right)^{\mu}f(z)=\frac{1}{2\pi i}\left(\frac{1-\nu}{B(\nu)}\right)^{\mu}\int_{|w-z|=\epsilon}\frac{f(w)}{w-z}\,\mathrm{d}w=\frac{1}{2\pi i}\left(\frac{1-\nu}{B(\nu)}\right)^{\mu}\int_H\frac{f(w)}{w-z}\,\mathrm{d}w,\]
and thus the $f(z)$ term from \eqref{IAB:Cdef} can be incorporated into the large integral.
\end{proof}

Now we can define an analytic continuation of the iterated AB differintegral to the whole complex plane for $\nu$ (having already achieved the whole complex plane for $\mu$ in Lemma \ref{Lem:IABintAC}). It follows from Lemma \ref{Lem:MML2conv} that the formula established in Theorem \ref{Thm:IABCauchy} defines an analytic function of $\nu\in\mathbb{C}\backslash\mathbb{R}$, while we showed in Lemma \ref{Lem:IABintAC} that the equivalent formula \eqref{IAB:Cdef} is an analytic function of $\nu$ in the punctured right half-plane $\Real(\nu)>0,\nu\neq1$. Analyticity in $z$ is preserved throughout if the function $f$ itself is analytic. Therefore, we propose the following definition.

\begin{definition} \label{Def:EIAB}
Let $c$ be a fixed complex number and $f$ be a complex function which is analytic on an open star-domain $R$ centred at $c$.

The \textbf{extended iterated AB differintegral} $\prescript{IAB}{c}I^{\nu,\mu}_{z}f(z)$ is defined for any $\mu\in\mathbb{C}$, any $\nu\in\mathbb{C}\backslash(\{1\}\cup\mathbb{R}^-)$, and any $z\in R\backslash\{c\}$ by:
\begin{alignat}{2}
\label{EIAB:1} \prescript{IAB}{c}I^{\nu,\mu}_{z}f(z)&=\int_{c}^z\left(\frac{1-\nu}{B(\nu)}\right)^{\mu}f(y)\left[\delta(z-y)+\sum_{n=1}^{\infty}\frac{\binom{\mu}{n}(1-\nu)^{-n}\nu^n}{\Gamma(n\nu)}(z-y)^{n\nu-1}\right]\,\mathrm{d}y,\quad&&\;\Real(\nu)>0; \\
\label{EIAB:2} \prescript{IAB}{c}I^{\nu,\mu}_{z}f(z)&=\frac{1}{2\pi i}\left(\frac{1-\nu}{B(\nu)}\right)^{\mu}\int_H\sum_{n=0}^{\infty}\frac{\Gamma(\mu+1)\Gamma(1-n\nu)}{\Gamma(n+1)\Gamma(\mu-n+1)}\left(\frac{\nu}{1-\nu}(w-z)^{\nu}\right)^n\frac{f(w)}{w-z}\,\mathrm{d}w,\quad&&\;\nu\in\mathbb{C}\backslash\mathbb{R}.
\end{alignat}
\end{definition}

\begin{proposition}
\label{Prop:EIAB:AC}
The extended iterated AB differintegral proposed in Definition \ref{Def:EIAB} is:
\begin{enumerate}
\item an analytic function of $z\in R\backslash\{c\}$, of $\mu\in\mathbb{C}$, and of $\nu\in\mathbb{C}\backslash\mathbb{R}^-_0$, provided $f$ and $B$ are analytic;
\item identical to the original formulae \eqref{IAB:RLdef}--\eqref{IAB:deltadef} in the case when $0<\nu<1$, $\mu\in\mathbb{R}$, and $c<z$ in $\mathbb{R}$.
\end{enumerate}
Therefore, it is an \textbf{analytic continuation} of the original definition to complex values of $z$, $\mu$, and $\nu$.
\end{proposition}

\begin{proof}
We showed in Lemma \ref{Lem:IABintAC} that the formula \eqref{EIAB:1}, being clearly equivalent to \eqref{IAB:Cdef}, is a valid extension of the original definition \eqref{IAB:2def} from $\mu\in\mathbb{R},0<\nu<1$ to $\mu,\nu\in\mathbb{C},\Real(\nu)>0$. And the three definitions \eqref{IAB:RLdef}--\eqref{IAB:deltadef} are already known to be equivalent.

Meanwhile, it follows from Lemma \ref{Lem:MML2conv} that the formula \eqref{EIAB:2} defines an analytic function of $\mu\in\mathbb{C}$ and $\nu\in\mathbb{C}\backslash\mathbb{R}$, and Theorem \ref{Thm:IABCauchy} tells us that this formula agrees with \eqref{EIAB:1} on the region $\Real(\nu)>0,\nu\not\in\mathbb{R}$.

The above results were proved for real $z$, but analyticity in $z\in R\backslash\{c\}$ is preserved throughout if the function $f$ itself is analytic.
\end{proof}

\begin{remark} \label{Rem:IABnodelta}
Once again, the advantage of the complex-analytic approach over the original definition is that it enables us to use a simpler kernel function.

The formula \eqref{EIAB:1} which was used in the real case requires taking the convolution of $f$ with a kernel which is not actually a function but rather a distribution, written in terms of the Dirac delta. The only alternative is the equivalent of using \eqref{IAB:2def} instead of \eqref{IAB:deltadef}: moving the first term outside of the integral, thus avoiding the use of the Dirac delta but requiring the definition to be a linear combination of two expressions rather than a single integral. In either case, the $n=0$ term must be treated separately.

By contrast, the Cauchy-type formula \eqref{EIAB:2} which can be used in the complex case is expressed as the convolution of $f$ with a single well-defined kernel function. As before (Remark \ref{Rem:ABnodelta}), the reason is that in complex analysis any function $f(z)$ can be written as a convolution of itself with $\frac{1}{z}$, whereas in real analysis the delta function must be invoked to achieve the same effect.
\end{remark}

\section{Properties and relations of the extended operators} \label{Sec:consequences}

Armed with the definitions that were established in the previous section, we can now prove some important facts about the extended versions of the AB and iterated AB operators.

\begin{theorem}[Inversion relations]
\label{Thm:EABcompos}
Let $c$ be a fixed complex number and $f$ be a complex function which is analytic on an open star-domain $R$ centred at $c$. For any $\nu\in\mathbb{C}\backslash\mathbb{R}^-$ and any $z\in R\backslash\{c\}$, we have the following composition relations between the extended AB derivatives and integrals:
\begin{align}
\label{EABcompos1} \prescript{ABR}{c}D^{\nu}_z\Big(\prescript{AB}{c}I^{\nu}_zf(z)\Big)&=f(z); \\
\label{EABcompos2} \prescript{AB}{c}I^{\nu}_z\Big(\prescript{ABR}{c}D^{\nu}_zf(z)\Big)&=f(z); \\
\label{EABcompos3} \prescript{AB}{c}I^{\nu}_z\Big(\prescript{ABC}{c}D^{\nu}_zf(z)\Big)&=f(z)-f(c).
\end{align}
However, the semigroup property for extended AB differintegrals does not hold in general. For $\mu,\nu\in\mathbb{C}\backslash\mathbb{R}^-$:
\begin{align*}
\prescript{AB}{c}I^{\mu}_z\Big(\prescript{AB}{c}I^{\nu}_zf(z)\Big)&\neq\prescript{AB}{c}I^{\mu+\nu}_zf(z); \\
\prescript{ABR}{c}D^{\mu}_z\Big(\prescript{ABR}{c}D^{\nu}_zf(z)\Big)&\neq\prescript{ABR}{c}D^{\mu+\nu}_zf(z); \\
\prescript{ABC}{c}D^{\mu}_z\Big(\prescript{ABC}{c}D^{\nu}_zf(z)\Big)&\neq\prescript{ABC}{c}D^{\mu+\nu}_zf(z).
\end{align*}
\end{theorem}

\begin{proof}
By Proposition \ref{Prop:EAB:AC} and Proposition \ref{Prop:EABI:AC}, the identities \eqref{EABcompos1}--\eqref{EABcompos3} follow directly by analytic continuation from the original and analogous results \eqref{ABcompos1}--\eqref{ABcompos3}, which are known to be true by Proposition \ref{Prop:ABcompos}.
\end{proof}

\begin{theorem}[Commutativity relations]
\label{Thm:EABcommut}
Let $c$ be a fixed complex number and $f$ be a complex function which is analytic on an open star-domain $R$ centred at $c$. For any $\mu,\nu\in\mathbb{C}$ such that $\mu,\nu,\mu+\nu\not\in\mathbb{R}^-$, and for any $z\in R\backslash\{c\}$, we have the following composition relations between the extended ABR derivatives and extended AB integrals:
\begin{align}
\label{EABcommut1} \prescript{AB}{c}I^{\mu}_z\Big(\prescript{AB}{c}I^{\nu}_zf(z)\Big)&=\prescript{AB}{c}I^{\nu}_z\Big(\prescript{AB}{c}I^{\mu}_zf(z)\Big); \\
\label{EABcommut2} \prescript{ABR}{c}D^{\mu}_z\Big(\prescript{ABR}{c}D^{\nu}_zf(z)\Big)&=\prescript{ABR}{c}D^{\nu}_z\Big(\prescript{ABR}{c}D^{\mu}_zf(z)\Big); \\
\label{EABcommut3} \prescript{ABR}{c}D^{\nu}_z\Big(\prescript{AB}{c}I^{\mu}_zf(z)\Big)&=\prescript{AB}{c}I^{\mu}_z\Big(\prescript{ABR}{c}D^{\nu}_zf(z)\Big).
\end{align}
\end{theorem}

\begin{proof}
By Proposition \ref{Prop:EAB:AC} and Proposition \ref{Prop:EABI:AC}, the identities \eqref{EABcommut1}--\eqref{EABcommut3} follow directly by analytic continuation from the original and analogous results \eqref{ABcommut1}--\eqref{ABcommut3}, which are known to be true by Proposition \ref{Prop:ABcommut}.
\end{proof}

\begin{theorem}
\label{Thm:EIAB:EAB}
Let $c$ be a fixed complex number and $f$ be a complex function which is analytic on an open star-domain $R$ centred at $c$. For any $\mu,\nu\in\mathbb{C}$ such that $\nu\not\in\mathbb{R}^-$ and $\nu\neq1$, for any $z\in R\backslash\{c\}$, and for any $n\in\mathbb{N}$, we have the following relations between the extended iterated AB differintegral, the extended AB integral, and the extended ABR derivative:
\begin{align*}
\prescript{IAB}{c}I^{0,\mu}_{z}f(z)&=f(z)=\prescript{IAB}{c}I^{\nu,0}_{z}f(z); \\
\prescript{IAB}{c}I^{\nu,n}_{z}f(z)&=\Big(\prescript{AB}{c}I^{\nu}_z\Big)^nf(z); \\
\prescript{IAB}{c}I^{\nu,-n}_{z}f(z)&=\Big(\prescript{ABR}{c}I^{\nu}_z\Big)^nf(z).
\end{align*}
\end{theorem}

\begin{proof}
All of these are proved in \cite{fernandez-baleanu2} for the original AB and iterated AB operators. By Proposition \ref{Prop:EAB:AC} and Proposition \ref{Prop:EABI:AC} and Proposition \ref{Prop:EIAB:AC}, the identities stated here for the extended operators follow directly by analytic continuation.
\end{proof}

\begin{theorem}[Semigroup property]
\label{Thm:EIABsemigroup}
Let $c$ be a fixed complex number and $f$ be a complex function which is analytic on an open star-domain $R$ centred at $c$. For any $\mu,\nu,\rho\in\mathbb{C}$ such that $\nu\not\in\mathbb{R}^-$ and $\nu\neq1$, and for any $z\in R\backslash\{c\}$, we have the following semigroup property in the second variable for extended iterated AB differintegrals:
\[\prescript{IAB}{c}I^{\nu,\mu}_{z}\Big(\prescript{IAB}{c}I^{\nu,\rho}_{z}f(z)\Big)=\prescript{IAB}{c}I^{\nu,\mu+\rho}_{z}f(z).\]
\end{theorem}

\begin{proof}
By Proposition \ref{Prop:EIAB:AC}, this identity follows directly by analytic continuation from the analogous result for the original iterated AB differintegral, which is established in Proposition \ref{Prop:IABsemigroup}.
\end{proof}

In a similar manner, many other properties of the AB and iterated AB models of fractional calculus can be straightforwardly extended to the complex-analytic extensions of those models which are defined in the current work. Essentially, what we have called the extended AB model and the extended iterated AB model are the natural way to define AB derivatives, AB integrals, and iterated AB differintegrals to complex orders of differintegration.

\begin{remark}
It is important to note that the AB derivatives and the AB integral are \textbf{not} two ways of defining the same operator with different signs.

In the case of the Riemann--Liouville derivative and integral, writing $\prescript{RL}{c}D^{\nu}_xf(x)$ and $\prescript{RL}{c}I^{-\nu}_xf(x)$ are equivalent. The analytic continuation of the Riemann--Liouville integral to negative values of $\nu$ yields precisely what we call the Riemann--Liouville derivative, its inverse operator.

In the Atangana--Baleanu model, this is not the case. The ABR derivative and the AB integral have fundamentally different structures from each other, their only relation being that they are inverse operators. This was not apparent in the original literature on the AB operators, which only treated $\prescript{ABR}{c}D^{\nu}_xf(x)$ and $\prescript{AB}{c}I^{\nu}_xf(x)$ in the case of $0<\nu<1$. But now that we have the analytic continuations of these operators to all values of $\nu$, we can clearly see that $\prescript{AB}{c}I^{-\nu}_zf(z)$ is not equal to either $\prescript{ABR}{c}D^{\nu}_zf(z)$ or $\prescript{ABC}{c}D^{\nu}_zf(z)$.

One way to unify the ABR derivative and the AB integral into a single generalised operator is given by the iterated AB differintegral (Definition \ref{Def:IAB}), specifically the fact that putting $\mu=1$ gives the AB integral and putting $\mu=-1$ gives the ABR derivative. However, this unification requires a second parameter of differintegration. It is not possible to recover the ABR derivative simply by negating the order of the AB integral.

We note also that this issue relates to the issue of inversion of differintegral operators which was discussed in a generalised context in \cite{fernandez-ozarslan-baleanu}.
\end{remark}

Finally, we validate the extended definitions for AB integrals and derivatives by verifying how they apply to some elementary functions, namely generalised power functions and exponential functions. These are often the first functions to check for a new fractional differintegral operator.

\begin{proposition}
The extended AB integral and extended AB derivatives of a general power function are as follows:
\begin{align*}
\prescript{AB}{c}I^{\nu}_z\left((z-c)^{\alpha}\right)&=\frac{(z-c)^{\alpha}}{B(\nu)}\left(1-\nu+\nu(z-c)^{\nu}\cdot\frac{\Gamma(\alpha+1)}{\Gamma(\alpha+\nu+1)}\right); \\
\prescript{ABR}{c}D^{\nu}_z\left((z-c)^{\alpha}\right)&=\prescript{ABC}{c}D^{\nu}_z\left((z-c)^{\alpha}\right)=\frac{B(\nu)}{1-\nu}\cdot(z-c)^{\alpha}\Gamma(\alpha+1)E_{\nu,\alpha+1}\left(\frac{-\nu}{1-\nu}(z-c)^{\nu}\right),
\end{align*}
where the parameters $\alpha,\nu\in\mathbb{C}$ satisfy $\Real(\alpha)>-1$ and, for the extended ABR derivatives, $\nu\not\in\mathbb{R}^-$.
\end{proposition}

\begin{proof}
The result for extended AB integrals follows directly from the definition \eqref{EABint}, since the Riemann--Liouville differintegral of a power function is well known. For $\Real(\alpha)>-1$ and for all $\nu\in\mathbb{C}$, we have
\begin{equation}
\label{RLpower}
\prescript{RL}{c}I^{\nu}_z\big((z-c)^{\alpha}\big)=\frac{\Gamma(\alpha+1)}{\Gamma(\alpha+\nu+1)}(z-c)^{\alpha+\nu},
\end{equation}
and therefore, substituting \eqref{RLpower} into \eqref{EABint}, we have
\[\prescript{AB}{c}I^{\nu}_z\big((z-c)^{\alpha}\big)=\frac{1-\nu}{B(\nu)}(z-c)^{\alpha}+\frac{\nu}{B(\nu)}\prescript{RL}{c}I^{\nu}_z\big((z-c)^{\alpha}\big)=\frac{1-\nu}{B(\nu)}(z-c)^{\alpha}+\frac{\nu}{B(\nu)}\cdot\frac{\Gamma(\alpha+1)}{\Gamma(\alpha+\nu+1)}(z-c)^{\alpha+\nu},\]
which yields the desired result for extended AB integrals.

For the extended AB derivatives, we first assume $\Real(\nu)>0$ and use the series formulae from Proposition \ref{Prop:ABseries}, which are known to be valid on the domain $\Real(\nu)>0$ by Lemma \ref{Lem:Cseries}. We also use \eqref{RLpower} again for the Riemann--Liouville differintegrals appearing in each term of the series.
\begin{align*}
\prescript{ABR}{c}D^{\nu}_z\big((z-c)^{\alpha}\big)&=\frac{B(\nu)}{1-\nu}\sum_{n=0}^{\infty}\big(\tfrac{-\nu}{1-\nu}\big)^n\prescript{RL}{c}I^{n\nu}_z\big((z-c)^{\alpha}\big) \\
&=\frac{B(\nu)}{1-\nu}\sum_{n=0}^{\infty}\big(\tfrac{-\nu}{1-\nu}\big)^n\frac{\Gamma(\alpha+1)}{\Gamma(\alpha+n\nu+1)}(z-c)^{\alpha+n\nu} \\
&=\frac{B(\nu)}{1-\nu}\cdot(z-c)^{\alpha}\Gamma(\alpha+1)\sum_{n=0}^{\infty}\frac{\left[\tfrac{-\nu}{1-\nu}(z-c)^{\nu}\right]^n}{\Gamma(n\nu+\alpha+1)} \\
&=\frac{B(\nu)}{1-\nu}\cdot(z-c)^{\alpha}\Gamma(\alpha+1)E_{\nu,\alpha+1}\left(\tfrac{-\nu}{1-\nu}(z-c)^{\nu}\right).
\end{align*}
This shows the desired result for the ABR case. In the ABC case, we obtain exactly the same series:
\begin{align*}
\prescript{ABC}{c}D^{\nu}_x\big((z-c)^{\alpha}\big)&=\frac{B(\nu)}{1-\nu}\sum_{n=0}^{\infty}\big(\tfrac{-\nu}{1-\nu}\big)^n\prescript{RL}{c}I^{n\nu+1}_z\frac{\mathrm{d}}{\mathrm{d}z}\big((z-c)^{\alpha}\big) \\
&=\frac{B(\nu)}{1-\nu}\sum_{n=0}^{\infty}\big(\tfrac{-\nu}{1-\nu}\big)^n\alpha\prescript{RL}{c}I^{n\nu+1}_z\big((z-c)^{\alpha-1}\big) \\
&=\frac{B(\nu)}{1-\nu}\sum_{n=0}^{\infty}\big(\tfrac{-\nu}{1-\nu}\big)^n\alpha\frac{\Gamma(\alpha)}{\Gamma(\alpha+n\nu+1)}(z-c)^{\alpha-1+n\nu+1} \\
&=\frac{B(\nu)}{1-\nu}\sum_{n=0}^{\infty}\big(\tfrac{-\nu}{1-\nu}\big)^n\frac{\Gamma(\alpha+1)}{\Gamma(\alpha+n\nu+1)}(z-c)^{\alpha+n\nu}.
\end{align*}
This is the same as the series found above for the ABR case, so the final result will be the same too.

Now we have proved the desired result for extended ABR and ABC derivatives, under the assumption that $\Real(\nu)>0$. The general case follows by analytic continuation, making use of Proposition \ref{Prop:EAB:AC}.
\end{proof}

\begin{proposition}
The extended AB integral and extended AB derivatives of an exponential function are as follows:
\begin{align*}
\prescript{AB}{-\infty}I^{\nu}_z\left(e^{\alpha z}\right)&=\frac{e^{\alpha z}}{B(\nu)}\left(1-\nu+\nu\alpha^{-\nu}\right); \\
\prescript{ABR}{-\infty}D^{\nu}_z\left(e^{\alpha z}\right)&=\prescript{ABC}{-\infty}D^{\nu}_z\left(e^{\alpha z}\right)=\frac{B(\nu)e^{\alpha z}}{1-\nu+\nu\alpha^{-\nu}},
\end{align*}
where, for the extended ABR derivatives, the parameters $\alpha,\nu\in\mathbb{C}$ satisfy $1-\nu+\nu\alpha^{-\nu}\neq0$, $\alpha\neq0$, and $\nu\not\in\mathbb{R}^-$.
\end{proposition}

\begin{proof}
Again, the result for extended AB integrals follows directly from the definition \eqref{EABint}. For $\alpha\in\mathbb{C}\backslash\{0\}$ and $\nu\in\mathbb{C}$, and with the constant of differintegration $c=-\infty$, the Riemann--Liouville differintegral of an exponential function is given by the following well-known formula:
\begin{equation}
\label{RLexp}
\prescript{RL}{-\infty}I^{\nu}_z\big(e^{\alpha z}\big)=\alpha^{-\nu}e^{\alpha z}.
\end{equation}
Substituting \eqref{RLexp} into \eqref{EABint} gives:
\[\prescript{AB}{-\infty}I^{\nu}_z\big(e^{\alpha z}\big)=\frac{1-\nu}{B(\nu)}\big(e^{\alpha z}\big)+\frac{\nu}{B(\nu)}\prescript{RL}{-\infty}I^{\nu}_z\big(e^{\alpha z}\big)=\frac{1-\nu}{B(\nu)}\cdot e^{\alpha z}+\frac{\nu}{B(\nu)}\cdot\alpha^{-\nu}e^{\alpha z},\]
which yields the desired result for extended AB integrals.

We know from Theorem \ref{Thm:EABcompos} that the extended AB integral and extended ABR derivative are inverse to each other. Thus, since $\frac{1-\nu+\nu\alpha^{-\nu}}{B(\nu)}$ does not depend on $z$, we could deduce the result for extended ABR derivatives directly from that for extended AB integrals. However, we shall also provide a proof using the series formula, which works equally well for the ABR and ABC cases.
\begin{align*}
\prescript{ABR}{-\infty}D^{\nu}_z\big(e^{\alpha z}\big)&=\frac{B(\nu)}{1-\nu}\sum_{n=0}^{\infty}\big(\tfrac{-\nu}{1-\nu}\big)^n\prescript{RL}{-\infty}I^{n\nu}_z\big(e^{\alpha z}\big)=\frac{B(\nu)}{1-\nu}\sum_{n=0}^{\infty}\big(\tfrac{-\nu}{1-\nu}\big)^n\alpha^{-n\nu}e^{\alpha z} \\
&=\frac{B(\nu)}{1-\nu}\cdot e^{\alpha z}\sum_{n=0}^{\infty}\left(\frac{-\nu}{1-\nu}\cdot\alpha^{-\nu}\right)^n \\
&=\frac{B(\nu)e^{\alpha z}}{1-\nu}\cdot\frac{1}{1-\frac{-\nu}{1-\nu}\cdot\alpha^{-\nu}}=\frac{B(\nu)e^{\alpha z}}{1-\nu+\nu\alpha^{-\nu}}; \\
\prescript{ABC}{-\infty}D^{\nu}_z\big(e^{\alpha z}\big)&=\frac{B(\nu)}{1-\nu}\sum_{n=0}^{\infty}\big(\tfrac{-\nu}{1-\nu}\big)^n\prescript{RL}{-\infty}I^{n\nu+1}_z\frac{\mathrm{d}}{\mathrm{d}z}\big(e^{\alpha z}\big)=\frac{B(\nu)}{1-\nu}\sum_{n=0}^{\infty}\big(\tfrac{-\nu}{1-\nu}\big)^n\alpha\prescript{RL}{-\infty}I^{n\nu+1}_z\big(e^{\alpha z}\big) \\
&=\frac{B(\nu)}{1-\nu}\sum_{n=0}^{\infty}\big(\tfrac{-\nu}{1-\nu}\big)^n\alpha\cdot\alpha^{-n\nu-1}e^{\alpha z}=\frac{B(\nu)}{1-\nu}\sum_{n=0}^{\infty}\big(\tfrac{-\nu}{1-\nu}\big)^n\alpha^{-n\nu}e^{\alpha z} \\
&=\prescript{ABR}{-\infty}D^{\nu}_z\big(e^{\alpha z}\big)=\frac{B(\nu)e^{\alpha z}}{1-\nu+\nu\alpha^{-\nu}}.
\end{align*}
The series formulae are valid for $\Real(\nu)>0$, and for the convergence of the infinite geometric series, we also needed the assumption that $\left|\frac{-\nu}{1-\nu}\cdot\alpha^{-\nu}\right|<1$. For a given $\alpha\in\mathbb{C}\backslash\{0\}$, these two inequalities will both be valid for $\nu$ in some small open semicircle near zero. Thus, by analytic continuation, the results are valid for any $\nu\in\mathbb{C}\backslash\mathbb{R}^-$ such that the denominator is nonzero.
\end{proof}

\section{Conclusions} \label{Sec:conclusions}

In this paper, we have proposed an extension of the Atangana--Baleanu model of fractional calculus, defining AB derivatives and integrals to order which may be not only in the real interval $[0,1]$ but anywhere in the complex plane. By analogy with the Cauchy-type complex integral formula for Riemann--Liouville fractional differintegrals, we found a formula for the extended Atangana--Baleanu operators in terms of integration around a closed complex contour, which is often more useful in complex-analytic applications than the original integral along a straight line-segment.

We demonstrated the naturality of our approach by proving rigorously that the extended operators proposed here are in fact analytic continuations of the original AB derivatives and integrals. This enabled us to prove various important properties of the extended operators by direct analytic continuation of the corresponding results for the original real-order AB operators.

As well as considering the standard operators of Atangana--Baleanu fractional calculus (the AB integral, the ABR derivative, and the ABC derivative), we also studied a recent generalisation known as the iterated AB differintegral: a fractional-calculus operator which takes two parameters and includes both the AB integral and the ABR derivative, as well as their iterations, in a single unified operator. This operator too can be extended to complex orders, and we proved its properties and naturality in a similar way as for the AB derivatives and integrals.

Complex analysis is a major part of mathematics in general, but its applications in fractional calculus have largely been overlooked so far. Our formulation of the important Atangana--Baleanu model in terms of complex analysis will surely be useful in future investigations and applications. In particular, for now, the complex-analytic approach enables several operators, originally written as linear combinations of functions and integrals, or as integrals with Dirac-delta kernels, to be expressed more simply and elegantly as convolution integrals with a single kernel function.

\section*{Conflicts of interest}

There are no conflicts of interest to this work.

\end{document}